\documentclass[12pt,reqno]{article}

\usepackage{amsmath}
\usepackage{amsthm}
\usepackage{amssymb}
\usepackage[T1]{fontenc}

\usepackage{fullpage}
\usepackage{mathrsfs}
\usepackage[usenames]{color}
\usepackage{amssymb}
\usepackage{amscd}
\usepackage{pstricks}

\usepackage{color}
\usepackage{fullpage}
\usepackage{float}

\usepackage{graphics,amsmath,amssymb}
\usepackage{amsthm}
\usepackage{amsfonts}
\usepackage{latexsym}

\usepackage{pstricks}

\usepackage{url}

\begin{document}

\theoremstyle{plain}
\newtheorem{theorem}{Theorem}
\newtheorem{corollary}[theorem]{Corollary}
\newtheorem{lemma}[theorem]{Lemma}
\newtheorem{proposition}[theorem]{Proposition}

\theoremstyle{definition}
\newtheorem{definition}[theorem]{Definition}
\newtheorem{example}[theorem]{Example}
\newtheorem{conjecture}[theorem]{Conjecture}

\theoremstyle{remark}
\newtheorem{remark}[theorem]{Remark}

\newtheorem {Congruence}{Congruence}

\title{Chromatic Properties of the Euclidean Plane}
\author{James Currie\\
Department of Mathematics and Statistics\\
University of Winnipeg\\
Winnipeg, Manitoba\\
Canada R3B 2E9\\
\texttt{j.currie@uwinnipeg.ca}\vspace{.1in}\\
Roger B. Eggleton\\
School of Mathematical and Physical Sciences,\\
University of Newcastle,\\
Callaghan, NSW 2308, Australia\\
\texttt{
roger.eggleton@newcastle.edu.au}}
\maketitle
\begin{abstract}Let $G$ be the unit distance graph in the plane. A well-known problem in combinatorial geometry is that of determining the chromatic number of $G$. It is known that $4\le \chi(G)\le 7$. The upper bound of 7 is obtained using tilings of the plane. The present paper studies two problems where we seek proper colourings of $G$, adding restrictions inspired by tilings:
\begin{enumerate}
\item Let $H(\epsilon)$ be the graph whose vertices are the points of ${\mathbb R}^2$, with an edge between two points if their distance lies in the interval $[1,1+\epsilon]$. We show that for small $\epsilon$, $0<\epsilon\le \frac{3\sqrt{2}}{4}-1$, we have $6\le \chi(H(\epsilon))\le 7$. Exoo showed that $5\le \chi(H(\epsilon))$ for small $\epsilon$. This was strengthened by
Grytczuk et al., who removed the restriction to small $\epsilon$.
\item Suppose that $G$ is properly coloured, but so that two solidly coloured regions meet along a straight line in some neighbourhood. Then at least 5 colours must be used. 
\end{enumerate}

\end{abstract}

The real plane is to be coloured so that any two points at unit distance receive different colours. How many colours are required?

This problem from \cite{debruijn} may be considered as a graph colouring problem. Let $G$ be the graph whose vertices are points of ${\mathbb R}^2$, with an edge between any pair of points at unit distance. We seek the chromatic number of $G$. As noted in \cite{moser}, because $G$ has the graph of Figure~\ref{kites} as a subgraph, at least 4 colours are required. A proper colouring of $G$ by 7 colours can be had as indicated in Figure~\ref{tiling}. Thus $4\le \chi(G)\le 7$. The problem of determining the exact chromatic number of $G$ is still open, and has become well known.
\begin{figure}
\begin{center}
\setlength{\unitlength}{.1in}
\begin{picture}(25,20)
\put(0,0){\line(1,0){35}}
\put(0,5){\line(1,0){35}}
\put(0,10){\line(1,0){35}}
\put(0,15){\line(1,0){35}}
\put(0,20){\line(1,0){35}}

\put(0,0){\line(0,1){5}}
\put(5,0){\line(0,1){5}}
\put(10,0){\line(0,1){5}}
\put(15,0){\line(0,1){5}}
\put(20,0){\line(0,1){5}}
\put(25,0){\line(0,1){5}}
\put(30,0){\line(0,1){5}}
\put(35,0){\line(0,1){5}}

\put(0,0){\line(0,1){5}}
\put(5,0){\line(0,1){5}}
\put(10,0){\line(0,1){5}}
\put(15,0){\line(0,1){5}}
\put(20,0){\line(0,1){5}}
\put(25,0){\line(0,1){5}}
\put(30,0){\line(0,1){5}}
\put(35,0){\line(0,1){5}}

\put(2.5,5){\line(0,1){5}}
\put(7.5,5){\line(0,1){5}}
\put(12.5,5){\line(0,1){5}}
\put(17.5,5){\line(0,1){5}}
\put(22.5,5){\line(0,1){5}}
\put(27.5,5){\line(0,1){5}}
\put(32.5,5){\line(0,1){5}}

\put(0,10){\line(0,1){5}}
\put(5,10){\line(0,1){5}}
\put(10,10){\line(0,1){5}}
\put(15,10){\line(0,1){5}}
\put(20,10){\line(0,1){5}}
\put(25,10){\line(0,1){5}}
\put(30,10){\line(0,1){5}}
\put(35,10){\line(0,1){5}}

\put(2.5,15){\line(0,1){5}}
\put(7.5,15){\line(0,1){5}}
\put(12.5,15){\line(0,1){5}}
\put(17.5,15){\line(0,1){5}}
\put(22.5,15){\line(0,1){5}}
\put(27.5,15){\line(0,1){5}}
\put(32.5,15){\line(0,1){5}}

\put(2,2){1}
\put(7,2){2}
\put(12,2){3}
\put(16.3,1.3){\rotatebox{45}{\tiny 1 unit}}
\put(16.5,1.5){\vector(-1,-1){1.5}}
\put(18.5,3.5){\vector(1,1){1.5}}
\put(22,2){5}
\put(27,2){6}
\put(32,2){7}

\put(4.5,7){4}
\put(9.5,7){5}
\put(14.5,7){6}
\put(19.5,7){7}
\put(24.5,7){1}
\put(29.5,7){2}

\put(2,12){6}
\put(7,12){7}
\put(12,12){1}
\put(17,12){2}
\put(22,12){3}
\put(27,12){4}
\put(32,12){5}

\put(4.5,17){2}
\put(9.5,17){3}
\put(14.5,17){4}
\put(19.5,17){5}
\put(24.5,17){6}
\put(29.5,17){7}

\put(20,0){$\underbrace{\hspace*{.75in}}$}
\put(20.5,-2.5){$\frac{3\sqrt{2}}{4}$  {\scriptsize units}}

\end{picture}
\end{center}
\caption{\protect\footnotesize  
A proper colouring of $G$, the unit distance graph in the plane}
\label{tiling}
\end{figure}
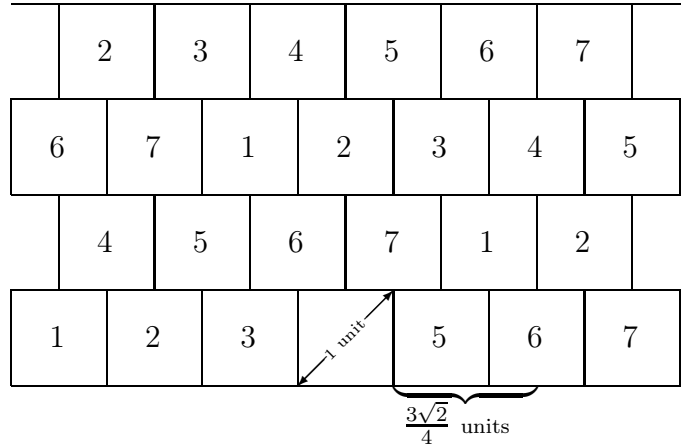

\begin{figure}
\begin{center}
\begin{pspicture}[shift=*](4,5)

\psset{unit=.3cm}
\qdisk(0,0){3pt}
\qdisk(10,0){3pt}
\qdisk(15,8.6){3pt}
\qdisk(5,8.6){3pt}

\qdisk(8.3,5.5){3pt}
\qdisk(-.6,10){3pt}
\qdisk(7.7,15.5){3pt}

\qline(0,0)(5,8.6)
\qline(5,8.6)(10,0)
\qline(0,0)(10,0)
\qline(5,8.6)(15,8.6)
\qline(10,0)(15,8.6)
\qline(0,0)(5,8.6)
\qline(0,0)(5,8.6)
\qline(0,0)(5,8.6)
\qline(0,0)(5,8.6)
\qline(0,0)(5,8.6)

\qline(0,0)(8.3,5.5)
\qline(0,0)(-.6,10)
\qline(8.5,5.5)(-.6,10)
\qline(7.7,15.5)(8.3,5.5)
\qline(7.7,15.5)(-.6,10)

\qline(7.7,15.5)(15,8.6)
\end{pspicture}
\end{center}
\caption{\protect\footnotesize  
A subgraph of $G$ requiring 4 colours}
\label{kites}
\end{figure}
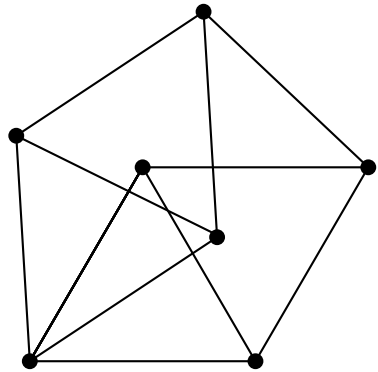

One method of attacking this problem \cite{benda,chilakamarri,eggleton,fischer1,fischer2,johnson2,woodall,zaks} has been to replace the plane ${\mathbb R}^2$ by the space $K^n$, where $K$ is some subfield of the reals, and $n\in{\mathbb N}$. For example, let $H$ be the induced subgraph of $G$ whose vertices have coordinates lying in ${\mathbb Q}(\sqrt{3})$. (Thus $K={\mathbb Q}(\sqrt{3})$, $n=2$.) Exactly 3 colours are required to properly colour $H$. Although $G$ is connected, $H$ is not, one may study the connectivity of `distance graphs' such as $H$.

We note that in Figure~\ref{tiling}, the plane is coloured in solid regions, rather than with `clouds' of colour. Whether such solid colourings will produce a colouring of $G$ with the minimum number of colours is not clear; perhaps a minimal colouring of $G$ has each colour everywhere dense in the plane. The results of this paper may be thought of as relating to plane colouring with solid regions. We produce lower bounds for the following colouring problems:

\begin{enumerate}
\item Let a small positive $\epsilon$ be given. Colour the plane so that any two points whose distance lies between 1 and $1+\epsilon$ are coloured differently. How many colours are required? 
\item Suppose that $G$ is properly coloured in such a way that two solidly coloured regions share a straight-line boundary. How many colours are required? 
\end{enumerate}

We provide lower bounds of 6 and 5 for problems 1 and 2 respectively. The first of these improves the bound $5\le H(\epsilon)$ found by Exoo and Grytczuk et al. \cite{exoo,grytczuk}.

In the colouring of Figure 2, the left-hand side and bottom of a square tile belong to the tile, but not the right-hand side or top; thus the monochromatic squares are translates of $[0,1)^2$. There is some `play' in this tiling; each row is shifted $3\sqrt{2}/4$ units to the left relative to the row below it. Note that $1<3\sqrt{2}/4\doteq 1.06$; thus no two points at distances in the interval $[1,1+\epsilon]$ receive the same colour, where $\epsilon=3\sqrt{2}/4-1$. Thus the colouring of Figure 2 is proper for both problems we consider (for small $\epsilon$), and an upper bound in both cases is 7.
\section{The plane distance graph connecting points at distances in $[1,1+\epsilon]$}
\begin{theorem}\label{epsilon} Fix $\epsilon$, $0<\epsilon<\frac{3\sqrt{2}}{4}-1$. Let $H$ be the graph whose vertices are the points of ${\mathbb R}^2$, with an edge between every pair of points whose distance lies in the interval $[1,1+\epsilon]$. Then $6\le \chi(G)\le 7$.
\end{theorem}
\begin{proof} Because of our observations regading the tiling of Figure~\ref{tiling}, it suffices to prove the lower bound. Let $c:{\mathbb R}^2\rightarrow \{1,2,3,\ldots, m\}$ be a proper colouring of $H$. Choose $\delta$, $\epsilon/2<\delta <\epsilon$, such that a rotation by $\theta=2\arcsin\frac{1+\delta}{2+\epsilon}$ has odd order $k$. 

Consider two points $P$ and $Q$ of ${\mathbb R}^2$ at distance $1+\epsilon/2$, and let $\phi$ be the rotation of ${\mathbb R}^2$ by $\theta$ about $P$.

Let us consider the induced subgraph $W(P)$ of $H$ with vertices $P, \phi(Q), \phi^2(Q),\ldots, \phi^{k-1}(Q)$ about $P$. Then $W(P)$ is a wheel centered at $P$ with `spokes' of length $1+\epsilon/2$, and a distance $1+\delta$ between adjacent points on the `rim'. The subgraph $R(P)$ of $H$ induced by the `rim' vertices $Q$, $\phi(Q)$, $\phi^2(Q),\ldots,\phi^{k-1}(Q)$ is an odd cycle, and $\chi(R(P))=3$. (See Figure~\ref{wheel}.) Colouring $c$ must assign at least 3 colours to $R(P)$. Further, each vertex of $R(P)$ is adjacent to every point of the closed ball $B(P,\epsilon/2)=\{X\in{\mathbb R}^2:|X-P|\le\epsilon/2\}$. Thus if $r$ is the number of colours assigned by $c$ to $B(P,\epsilon/2)$, then $c$ uses at least $3 + r$ colours to colour $H$.

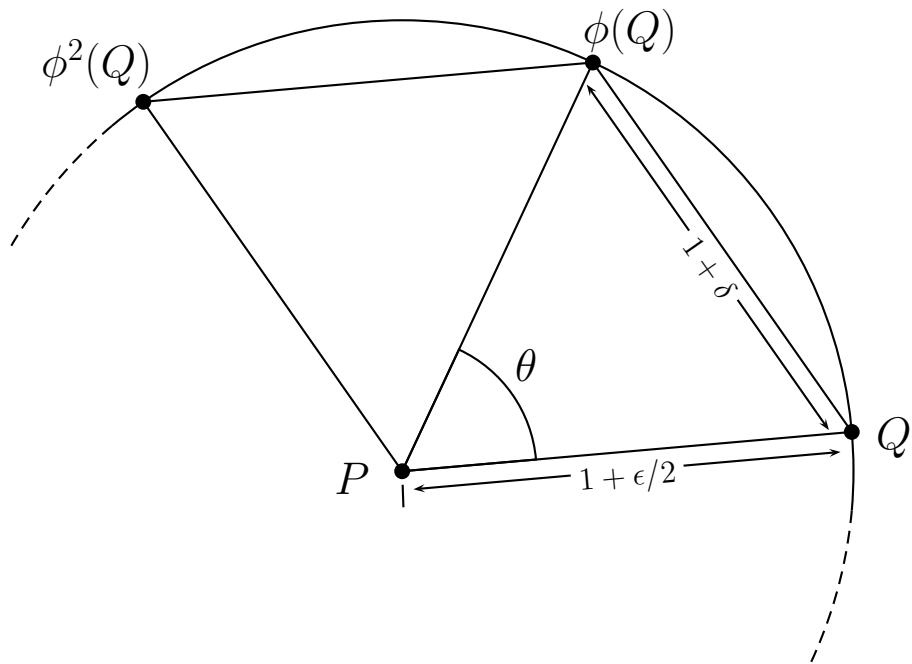
\begin{figure}
\begin{center}
\begin{pspicture}[shift=*](15,15)
\psset{unit=.6cm}
\SpecialCoor
\psarc[origin={10,5}](0,0){10}{-5}{130}
\psarc[origin={10,5},linestyle=dashed](0,0){10}{-25}{-5}
\psarc[origin={10,5},linestyle=dashed](0,0){10}{130}{150}
\pscircle*[origin={10,5}](10;5){3pt}
\pscircle*[origin={10,5}](10;65){3pt}
\pscircle*[origin={10,5}](10;125){3pt}
\pscircle*[origin={10,5}](0,0){3pt}
\psline[origin={10,5}]{<->}(9.5;5.5)(9.5;64.5)

\put(8.5,4.5){\Large $P$}
\put(20.5,5.5){\Large $Q$}
\put(14,14.5){\Large $\phi(Q)$}
\put(2,13.75){\Large $\phi^2(Q)$}
\put(12.5,7){\Large $\theta$}
\pswedge[origin={10,5}](0,0){3}{5}{65}

\rput{305}(16.8,9.5){\psframebox*{$1+\delta$}}

\psline[origin={10,5}]{}(0;0)(10;5)
\psline[origin={10,5}]{}(0;0)(10;65)
\psline[origin={10,5}]{}(0;0)(10;125)
\psline[origin={10,5}]{}(10;65)(10;5)
\psline[origin={10,5}]{}(10;65)(10;125)
\psline[origin={10,5}]{}(0.3;272)(.8;272)
\psline[origin={10.25,4.6}]{<->}(0;0)(9.5;5)

\rput{5}(15,4.85){\psframebox*{$1+\epsilon/2$}}
\end{pspicture}
\end{center}
\caption{\protect\footnotesize  
The induced subgraph $W(P)$ of $H$}
\label{wheel}
\end{figure}

So far, we have not worried how to pick $P$. We shall now show that $P$ can be chosen so that $r\ge 3$. This will complete our proof.

Suppose $\gamma$ is fixed, $0<\gamma<\epsilon$. Consider the plane lattice consisting of all points with coordinates of the form $i\gamma(1,0)+j\gamma(1/2,\sqrt{3}/2), i,j\in {\mathbb Z}$. We get the {\bf triangular lattice graph} $L(\gamma)$ by connecting every pair of points in this lattice at distance $\gamma$.

Colouring $c$ induces a natural partition of $L(\gamma)$ into induced subgraphs $L_1, L_2,\ldots, L_m$ where $V(L_i)=\{u\in V(L):c(u)=i\}$. If $u\in L(\gamma)$ and $c(u)=i$, then by the monochromatic component of $L(\gamma)$ containing $u$, we mean the connected component of $L_i$ containing $u$. 
\begin{lemma}\label{bounded}
Let $u\in L$; let $M$ be the monochromatic component of $L(\gamma)$ containing $u$. Then if $w\in M$, the distance $\rho(u,w)$ between $u$ and $w$ is less than $1$.
\end{lemma}
\noindent{\em Proof of Lemma~\ref{bounded}:}
Our proof works by induction on the (graph) length of a $uw$ path in $M$.

\noindent {\bf Length 0:} If there is a path from $u$ to $w$ of length 0, then $u=w$, and $\rho(u,w)=0<1$. 

\noindent {\bf Length $m$:} Assume that if there is a path from $u$ to $w$ of length m, tehn $\rho(u,w)<1$. 

\noindent {\bf Length $m+1$:} Suppose that a $uw$ path in $M$ is given by $u=u_1,u_2,\ldots, u_m, u_{m+1}=w$. By our induction hypothesis, $\rho(u,u_m)<1$. But now $\rho(u,w)\le\rho(u,u_m)+\rho(u_m,u_{m+1}<1+\gamma<1+\epsilon$. Since $c(u)=c(w)$, we cannot have $1\le \rho(u,w)\le1+\epsilon$. Thus $\rho(u,w)<1$, as claimed. 

It follows that diam $M \le 1$. ({\em End of proof of Lemma~\ref{bounded}})

If $C$ is any cycle in $L(\gamma)$, let {\bf int} $C$ denote the set of vertices of $L(\gamma)$ which are in the interior of $C$, and let $V(C)$ denote the set of all vertices of $C$ itself. Given any finite subgraph $M$ of $L(\gamma)$, we define a {\bf separating cycle} for $M$ to be any cycle of $L(\gamma)$ which contains $M$ in its interior, that is, any cycle in $L(\gamma)$ such that all vertices of $M$ belong to int $C$. By considering very large cycles in $L(\gamma)$, it is clear that $M$ does have a separating cycle.

Among all separating cycles for $M$, a {\bf thin} separating cycle is one which minimizes the number of vertices in int $C$. Among all thin cycles for $M$, a {\bf minimal} separating cycle is one which minimizes the number of vertices in $V(C)$.

\begin{lemma}\label{thin}For any real $\gamma>0$, let $M$ be any non-empty finite connected subgraph of the triangular lattice graph $L(\gamma)$. Then $M$ has a unique minimal separating cycle $C$, and if two vertices are adjacent in $C$, there is some vertex of $M$ which is adjacent to both of them.
\end{lemma}

\noindent{\em Proof of Lemma~\ref{thin}:}

Consider a separating cycle $C$ for $M$. Suppose that two vertices $v$ and $w$ are adjacent in $C$, but that no vertex of $M$ is adjacent to both of them. Just two vertices of $L(\gamma)$, say $x$ and $y$, are adjacent to both $v$ and $w$. Note that $x$, $w$ and $y$ are adjacent vertices on $B$, the 6-cycle bounding the closed (graph) neighbourhood of $v$. (See Figure~\ref{hex}.) If neither $x$ nor $y$ belongs to $C$, the other vertex of $C$ which is adjacent to $v$, say $u$, must be on the major arc of $B$ between $x$ and $y$. Thus the paths $uvw$ and $xvy$ cross at $v$, so $xvy$ crosses $c$ in only one point, namely $v$. The Jordan Curve Theorem therefore implies that one of $x$ and $y$ must be interior to $C$, the other exterior.

\begin{figure}
\begin{center}
\begin{pspicture}[shift=*](15,15)
\psset{unit=.6cm}
\SpecialCoor
\psline[origin={10,10},linestyle=dashed]{}(0;0)(5;0)
\psline[origin={10,10},linestyle=dashed]{}(0;0)(5;120)
\psline[origin={15,10},linestyle=dashed]{}(0;0)(5;120)
\psline[origin={5,10}]{}(0;0)(5;0)
\psline[origin={10,10}]{}(0;0)(5;60)
\psline[origin={8.75,12.2}]{}(0;0)(5;0)
\put(4.8,9.3){$u$}
\put(9.8,9.3){$v$}
\put(9.8,12.4){$S$}
\put(14.9,9.3){$y$}
\put(7.3,14.7){$x$}
\put(12.3,14.7){$w$}
\pscircle*[linecolor=white,origin={10,10}](0,0){3pt}
\pscircle[origin={10,10}](0,0){3pt}

\pscircle*[linecolor=white,origin={10,10}](5;0){3pt}
\pscircle[origin={10,10}](5;0){3pt}

\pscircle*[linecolor=white,origin={10,10}](5;180){3pt}
\pscircle[origin={10,10}](5;180){3pt}

\pscircle*[linecolor=white,origin={10,10}](5;120){3pt}
\pscircle[origin={10,10}](5;120){3pt}

\pscircle*[linecolor=white,origin={10,10}](5;240){3pt}
\pscircle[origin={10,10}](5;240){3pt}

\pscircle*[linecolor=white,origin={10,10}](5;300){3pt}
\pscircle[origin={10,10}](5;300){3pt}

\pscircle*[origin={10,10}](2.5;120){3pt}

\pscircle*[origin={15,10}](2.5;120){3pt}

\pscircle*[linecolor=white,origin={15,10}](5;120){3pt}
\pscircle[origin={15,10}](5;120){3pt}

\end{pspicture}
\end{center}
\caption{\protect\footnotesize  
The hexset centred at $v$}
\label{hex}
\end{figure}
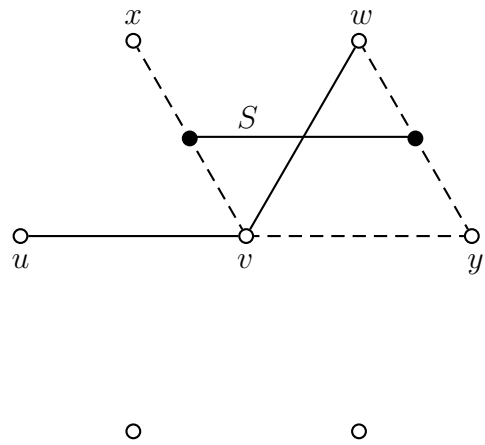

There are two possibilities: either exactly one of $x$ and $y$ belongs to int $C$, or at least one of $x$ and $y$ belongs to $V(C)$ and the other does not belong to int $C$.
\begin{enumerate}
\item Suppose one of $x$ and $y$, say $y$, is in the interior of $C$; the choice of $v$ and $w$ ensures that $y$ is not in $M$, so that we can form a separating cycle $C^{\prime\prime}$ from $C$ by replacing the path $vw$ by the path $vyw$. This construction ensures that int $C^{\prime\prime}$ is a proper subset of int $C$, since $y$ is in int $C^{\prime\prime}$, but not in int $C$.
\item Suppose that at least one of $x$ and $y$ belongs to $V(C)$ and the other does not belong to int $C$. 
Then $x$ cannot be adjacent in $C$ to both $v$ and $w$, for this would force $C$ to be a 3-cycle. Without loss of generality, we can assume that the labeling is chosen so that $x$ is not adjacent to $v$ in $C$. Extending this argument, we also assume that $y$ is not adjacent to $w$ in $C$. Let $S$ be a straight line segment with endpoints which are the midpoints of the edges $vx$ and $wy$. This crosses $C$ in precisely one point, namely, the midpoint of the edge $vw$. The Jordan Curve Theorem therefore implies that one of the endpoints of $S$ is in the interior of $C$, and the other is in the exterior. Hence, exactly one of the edges $vx$ and $wy$ is in the interior of $C$; the other is in the exterior. Without loss of generality, we may suppose that $vx$ is in the interior of $C$. Note that $x$ is not in the interior of $C$, so $x$ must be in $V(C)$.Cycle $C$ contains exactly two paths with endpoints $v$ and $x$; together which $vx$, each forms a cycle, and the interiors of these cycles are contained in the interior of $C$. But no edge of $L(\gamma)$ crosses $vx$, and $M$ is a connected subgraph of $L(\gamma)$, lying in the interior of $C$. It follows that one of the new cycles, say $C'$, is a separating cycle for $M$. Once again, the construction ensures that int $C'$ is a subset of int $C$, though not necessarily a proper one. However, at least two edges of $C$ do not belong to $C'$, so at least one vertex of $C$ does not belong to $C'$. Therefore $V(C')$ is a proper subset of $V(C)$.
\end{enumerate}
So long as two vertices of a separating cycle $C$ for $M$ are not both adjacent to some vertex of $M$, we can use one of the above constructions to obtain a separating cycle $C'$ for $M$ for which either int $C'$ is a proper subset of int $C$, or $V(C')$ is a proper subset of $V(C)$. Since int $C$ and $V(C)$ are initially finite, the process terminates in a finite number of steps. We finish with some separating cycle $C'$ which has the property that any two vertices adjacent in $C'$ are both adjacent to some vertex of $M$. 

To see that $C'$ is uniquely determined, suppose to the contrary that two such cycles could arise. There are two possibilities. If one such cycle lies entirely in the interior of the other, then no vertex of the outer one could be adjacent to any vertex of $M$. Otherwise, the two cycles must intersect, and then one contains a path that lies outside the other, and no vertex on that path could be adjacent to any vertex of $M$. These contradictions show that $C'$ is unique.
({\em End of proof of Lemma~\ref{thin}})

We are now ready to show the existence of a point $P$ such that $c$ gives $B(P,\epsilon/2)$ at least 3 colours:

Let $\gamma=\epsilon/3$. Let $v$ be any point of $L(\gamma)$. Denote the monochromatic component of $L(\gamma)$ containing $v$ by $M_1$. Having defined $M_i$, let $N_i$ be the unique thin separating cycle containing $M_i$. If $N_i$ is monochromatic, let $M_{i+1}$ be the monochromatic component of $L(\gamma)$ containing $N_i$. Note that $N_i$ is connected. Now let $j$ be least such that $N_j$ is notmonochromatic. Such a $j$ exists, since the diameter of a connected monochromatic subgraph of $L(\gamma)$ is bounded.

Let $x$ and $y$ be neighbours in $N_j$ such that $c(x)\ne c(y)$. By Lemma~\ref{thin}, there is a vertex $z$ of $M_j$ such that $z$ is adjacent to both $x$ and $y$. Moreover, since $M_j$ is a monochromatic component, $c(z)\ne c(z)$, $c(y)\ne c(z)$. If we pick $P$ to be the point of ${\mathbb R}^2$ occupied by $x$, then $B(P,\epsilon/2)$ contains $x, y$ and $z$, and thus receives at least three colours from $c$.

In conclusion, $6\le \chi(H)\le 7$.
\end{proof}

\section{Colourings with solidly coloured regions sharing a straight-line boundary}
Recall that $G$ is the graph whose vertices are points of ${\mathbb R}^2$, with an edge between any pair of points at unit distance.
\begin{theorem}
Let $c$ be a proper colouring of $G$ with the following property:

There is a point $P$ of ${\mathbb R}^2$, a half-plane $H$ , and a neighbourhood $N$ of $P$ so that $P$ lies on the line $H\cap H'$, $c$ colours each point of the interior of $H\cap N$ with some colour $a$, and each point of the interior of $H'\cap N$ with some colour $b$, $b\ne a$. (See Figure~\ref{isos}.) Then $c$ uses at least 5 colours.
\end{theorem}

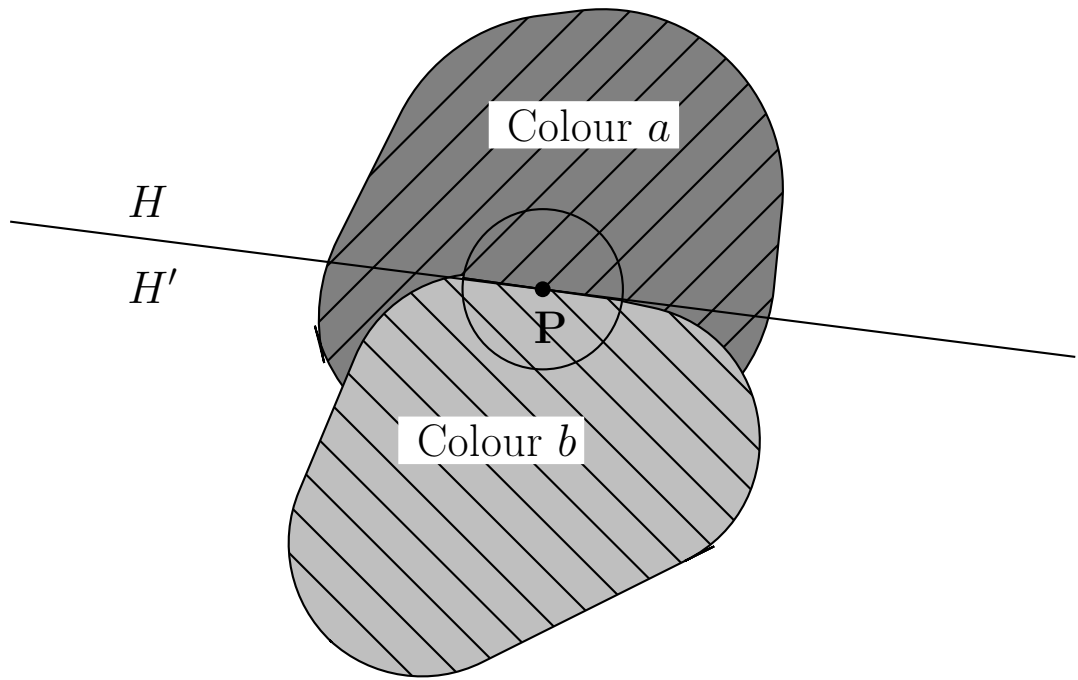
\begin{figure}
\begin{center}
\begin{pspicture}[shift=*](15,15)

\psset{unit=.6cm}
\pspolygon[linearc=4,fillstyle=hlines*,fillcolor=gray,hatchsep=12pt](10,10)(9,14)(12,20)(20,21)(19,11)
\pspolygon[linearc=3,fillstyle=vlines*,fillcolor=lightgray,hatchsep=12pt](11,5)(8,7.8)(11,15)(25,12)
\pswedge[fillstyle=hlines*,fillcolor=gray,hatchsep=12pt](14.2,14.2){1.8}{-8}{172}
\pswedge[fillstyle=vlines*,fillcolor=lightgray,hatchsep=12pt](14.2,14.2){1.8}{172}{-8}
\put(14,13){\bf \Large P}
\psline[]{}(2.4,15.7)(26,12.7)
\put(5,15.8){\bf \Large $H$}
\put(5,13.9){\bf \Large $H'$}
\qdisk(14.2,14.2){3pt}
\put(13,17.5){\psframebox*{ \Large Colour $a$}}
\put(11,10.5){\psframebox*{ \Large Colour $b$}}
\end{pspicture}
\end{center}
\caption{\protect\footnotesize  
Two solidly coloured regions sharing a straightline boundary in a neighbourhood}
\label{kites}
\end{figure}

\begin{proof}
Let $c$, $P$, $H$, $N$ be given as above, and let $\epsilon$ be chosen $0<\epsilon<1$, so that $B(P,\epsilon)$ lies inside $N$. Consider an isosceles triangle $T$ with base length $\epsilon$, equal sides length 1, such that the midpoint of the base of $T$ is $P$. Then the altitude of $T$ is $\sqrt{1-\epsilon^2/4}$. we note that the base of $T$ is contained in $B(P,\epsilon)$. In fact, suppose that $T(\delta)$ is an isosceles 
triangle with equal sides 1, altitude $\delta$, and the midpoint of its base at $P$. Then if $\sqrt{1-\epsilon^2/4}\le \delta<1$, the base of $T(\delta)$ will fall inside $B(P,\epsilon)$. (See Figure~\ref{isos}.)

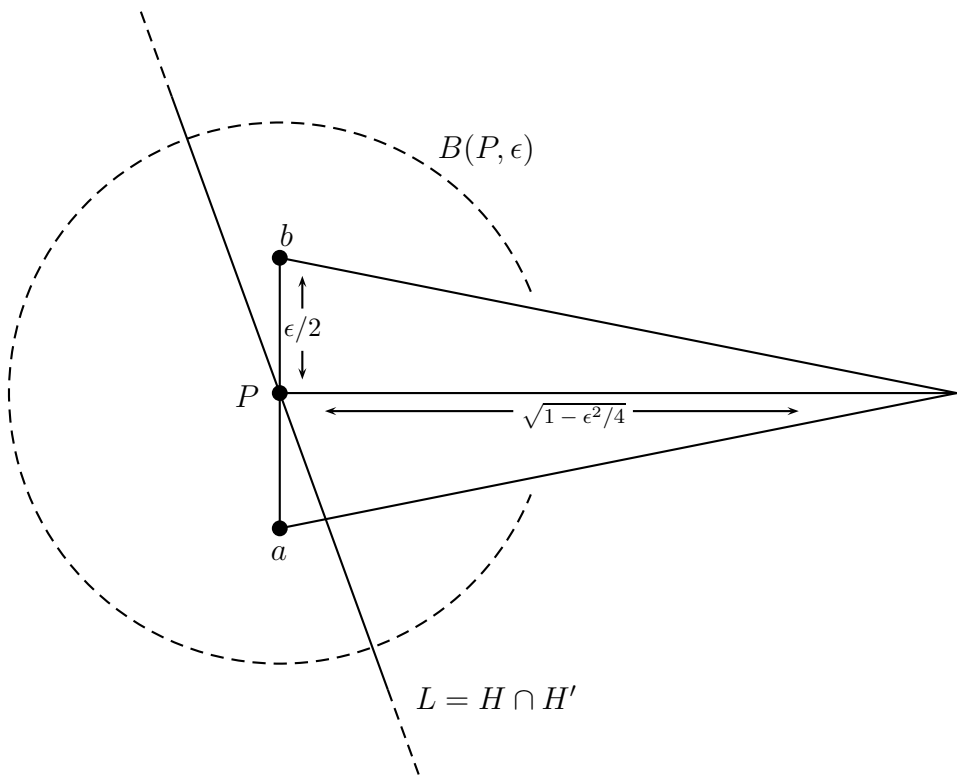
\begin{figure}
\begin{center}
\begin{pspicture}[shift=*](15,15)
\psset{unit=.6cm}
\SpecialCoor

\pscircle*[origin={7,10}](0,0){3pt}
\pscircle*[origin={7,10}](0,3){3pt}
\pscircle*[origin={7,10}](0,-3){3pt}

\psline[origin={7,10}]{}(0,3)(15,0)
\psline[origin={7,10}]{}(0,-3)(15,0)

\psline[origin={7,10}]{<->}(0.5,2.6)(0.5,0.3)
\rput(7.5,11.4){\psframebox*{\footnotesize$\epsilon/2$}}
\psline[origin={7,10}]{}(0,3)(0,-3)
\psline[origin={7,10}]{<->}(1,-.4)(11.5,-.4)
\rput(13.5,9.5){\psframebox*{\scriptsize $\sqrt{1-\epsilon^2/4}$}}
\psline[origin={7,10}]{}(0,0)(15,0)

\psline[origin={7,10}]{}(7;290)(7;110)
\psline[linestyle=dashed,origin={7,10}]{}(7;110)(9;110)
\psline[linestyle=dashed,origin={7,10}]{}(7;290)(9;290)
\put(10,3){$L=H\cap H'$}

\psarc[linestyle=dashed,origin={7,10}](0,0){6}{22}{338}

\put(10.5,15.2){$B(P,\epsilon)$}
\put(7,13.25){$b$}
\put(6,9.7){$P$}
\put(6.8,6.3){$a$}
\end{pspicture}
\end{center}
\caption{\protect\footnotesize  
The triangle $T$}
\label{isos}
\end{figure}

Suppose that $\delta$ is given, $\sqrt{1-\epsilon^2/4}\le\delta<1$. Let $L(\gamma)$ be the line $L=H\cap H'$. If the base of $T(\delta)$ does not coincide with $L(\gamma)$, then the base of $T(\delta)$ crosses $L(\gamma)$. In this case, the endpoints of the base of $T(\delta)$ will receive two different colours from $c$; one endpoint will be coloured $a$, the other $b$. The base of $T(\delta)$ coincides with $L(\gamma)$ only if the apex of $T(\delta)$ is on the perpendicular to $L(\gamma)$ at $P$.

Choose $\delta$, $\sqrt{1-\epsilon^2/4}< \delta <1$, such that a rotation by $\theta=2\arcsin(1/2\delta)$ has odd order $k$. Let $\phi$ be the rotation of ${\mathbb R}^2$ by $\theta$ about $P$. Choose a point $Q$ at distance $\delta$ from $P$ such that none of the $k$ points $Q$, $\phi(Q)$, $\phi^2(Q)\,\ldots,\phi^{k-1}(Q)$ lines on the perpendicular to $L(\gamma)$ through $P$.  Let us consider the induced subgraph $W(P)$ of $G$ with vertices $P$, $Q$, $\phi(Q)$, $\phi^2(Q)\,\ldots,\phi^{k-1}(Q)$. The graph $W(P)$ is a wheel centered at $P$ with`spokes' of length $\delta$, and a distance 1 between adjacent points on the `rim'. The subgraph $R(G)$ induced by the rim vertices $Q$, $\phi(Q)$, $\phi^2(Q)\,\ldots,\phi^{k-1}(Q)$ is an odd cycle, and $\chi(R(P))=3$. Choosing $T(\delta)$ to have $\phi^j(Q)$ as an apex, we see that $\phi^j(Q)$ is adjacent in $G$ to points coloured $a$ and points coloured $b$, for $j=0,1,,\ldots, k-1$.

It follows that $c$ uses at least five colours: $a$ and $b$, and at least three other colours for $R(P)$.
\end{proof}

\end{document}